\documentclass[reqno]{amsart}
\usepackage{amsmath,amsthm,amssymb,amsfonts,hyperref,bbm,mathabx}
\usepackage{enumerate}
\usepackage[english]{babel}
\usepackage[utf8]{inputenc}
\DeclareMathOperator{\n1}{\mathbbm{1}}

\DeclareMathOperator{\supp}{supp}

\DeclareMathOperator{\spec}{sp}

\DeclareMathOperator{\C}{\mathbb{C}}
\DeclareMathOperator{\R}{\mathbb{R}}
\DeclareMathOperator{\Z}{\mathbb{Z}}
\DeclareMathOperator{\T}{\mathbb{T}}

\DeclareMathOperator{\N}{\mathbb{N}}

\DeclareMathOperator{\B}{\mathcal{B}}

\newtheorem{thm}{Theorem}
\newtheorem*{thmA}{Theorem A}
\newtheorem*{thmB}{Theorem B}
\newtheorem*{thmC}{Theorem C}
\newtheorem*{thmD}{Theorem D}

\newtheorem{lem}[thm]{Lemma}
\newtheorem{prop}[thm]{Proposition}
\newtheorem{rmk}[thm]{Remark}
\newtheorem{cor}[thm]{Corollary}

\theoremstyle{definition}
\newtheorem{defn}[thm]{Definition}

\newtheorem{conj}{Conjecture}

\begin{document}
\title{On the growth of Fourier multipliers}

\author{Bat-Od Battseren}
\address{Université Côte d'Azur, LJAD, Nice F-06000, France}
\curraddr{}
\email{batoddd@gmail.com}
\urladdr{https://sites.google.com/view/batod-battseren/}
\thanks{}

\date{\today}

\dedicatory{}


\begin{abstract}
We define a sequence of functions, namely tame cuts, in the Fourier algebra $A(G)$ of a locally compact group $G$, that satisfies certain convergence and growth conditions. This new consideration allows us to give a group admitting a Fourier multiplier that is not completely bounded. Furthermore, we show that the induction map $MA(\Gamma)\rightarrow MA(G)$ is not always continuous. We also show how Liao's Property $(T_{Schur}, G, K)$ opposes tame cuts. Some examples are provided.
\end{abstract}

\maketitle
\tableofcontents

\section{Introduction}
Admitting a certain approximate identity in the Fourier algebra $A(G)$ of a locally compact group $G$ can be very useful property for answering analytical, algebraic, or geometrical questions on the locally compact group $G$. The most famous ones would be John von Neumann's amenability, and its generalizations: a-T-menability (also known as Haagerup's property), weak amenability, and approximation property. It is common that these properties consider an approximate identity that is bounded in a related topology. For example, a locally compact group $G$ is weakly amenable if there is an approximate identity in $A(G)$ consisting of compactly supported continuous functions that are uniformly bounded in $M_0A(G)$ (see \cite{dCH1985,haagerup2016,haagerup1989}). In the present paper, we consider a sequence of compactly supported Fourier multipliers that are generally unbounded in Fourier multiplier algebra $MA(G)$ but grow at polynomial rate.
\begin{defn}\label{def1} Let $G$ be a locally compact group, and $\ell$ a proper length function on $G$. A sequence $(\varphi_n)$ of compactly supported, continuous functions is called \textit{tame cuts} for $(G,\ell)$ if there exist constants $C,a\geq 0$ such that
\begin{align*}
\varphi_n|_{B_n} \equiv 1 \quad \text{and} \quad \|\varphi_n\|_{MA}\leq Cn^a \quad \text{for all} \quad n\in \N.
\end{align*} If in addition, we have $\|\varphi_n\|_{M_0A}\leq Cn^a$ for all $n\in \N$, we say that the sequence $(\varphi_n)$ is \textit{completely bounded tame cuts}.
\end{defn}

The main goal is to provide some applications of tame cuts, to show how tame cuts connect to other properties, and to provide groups with (completely bounded) tame cuts.

Our first result is the following theorem in which we provide a group admitting a Fourier multiplier that is not completely bounded. The proof uses the tame cuts.
\begin{thmA}
Let $\Gamma$ be a uniform lattice in $G=SL(3,\R)$. Then $M_0 A(\Gamma)$ is a proper subalgebra of $MA(\Gamma)$.
\end{thmA}
This result provides supporting examples for the following conjecture.
\begin{conj}\label{conj}
A discrete countable group $\Gamma$ is amenable if and only if all Fourier multipliers of $\Gamma$ are completely bounded, i.e. $M_0A(\Gamma) = MA(\Gamma)$.	
\end{conj}
The "only if" part is already known even for locally compact groups (see \cite{Nebbia1982,Losert1984}).
On the other hand, in  \cite{haagerup2010schur}, Haagerup-Steenstrup-Szwarc constructed an explicit Fourier multiplier of the free group $F_2$ of two generators that is not completely bounded. Also in \cite{bozejko1981remark}, Bo\.{z}ejko used infinite lacunary sets to prove the same result. Their result directly implies Theorem A since any lattice of $SL(3,\R)$ contains a copy of $F_2$. We note, however, that our proof for lattices of $SL(3,\R)$ differs for being not constructive and for using tame cuts.

Our next result is about the \textit{induction map} defined as \begin{align*}
\Phi:\ell^\infty (\Gamma) \rightarrow C_b (G),\quad \varphi\mapsto\widehat{\varphi} = \n1_{\Omega} * (\varphi \mu_\Gamma)*\widetilde{\n1}_{\Omega}.
\end{align*} Recall that this map is used to prove that the properties of amenability, a-T-menability, weak amenability, approximation property, and property (T) are inherited by lattices. This is achieved by showing its restrictions $\Phi: A(\Gamma)\rightarrow A(G)$ and $\Phi:M_0A(\Gamma)\rightarrow M_0A(G)$ are continuous. When $G$ is amenable, the map $\Phi:MA(\Gamma)\rightarrow MA(G)$ is also well defined and continuous. This is in general not true as the following theorem shows.
\begin{thmB}\label{thmB}
Let $\Gamma$ be a uniform lattice in $G=SL(3,\R)$. Then the map
\begin{align*}
\Phi:MA(\Gamma)\rightarrow MA(G), \quad \varphi\mapsto\widehat{\varphi} = \n1_{\Omega} * (\varphi \mu_\Gamma)* \widetilde{\n1}_{\Omega}
\end{align*} is not bounded.
\end{thmB}

In \cite{Lia16}, property $(T_{Schur},G,K)$ was introduced to illustrate that any of the known methods so far to prove Baum-Connes conjecture for a particular group does not work for some groups. We will show in the following theorem that  property $(T_{Schur},G,K)$ also opposes the existence of tame cuts.
\begin{thmC}
Let $H$ be an unbounded closed subgroup of a locally compact group $G$. Suppose that $H$ satisfies property $(T_{Schur},G,K,\ell)$ for a compact subgroup $K$ and a proper length function $\ell$ of $G$. Then $(H,\ell|_H)$ does not admit any $K$-bi-invariant tame cuts. 
\end{thmC}

For discrete groups, we also consider characteristic functions as it is connected to the rapid decay property.
\begin{defn}
Let $\Gamma$ be a discrete group, and $\ell$ a proper length function on $\Gamma$. If $(\varphi_n)$ is a sequence of characteristic functions that gives (completely bounded) tame cuts for $(\Gamma,\ell)$,  then we call $(\varphi_n)$ \textit{(completely bounded) characteristic tame cuts} for $(\Gamma,\ell)$.
\end{defn}

The behaviors of the operator norm of reduced group $C^*$-algebra $C^*_\lambda (\Gamma)$ are notoriously difficult to understand. For instance, Valette's conjecture on rapid decay for uniform lattices in higher rank Lie groups is still open. Different approaches to understand the operator norm are studied for some particular cases: when $G$ is amenable, we have $\|\lambda(f)\| = \|f\|_1$  for all non-negative functions $f\in C_c(G)$, and when $G$ is a connected simple Lie group, the Harish-Chandra spherical function $\phi_0$ satisfies $\|\lambda(f)\|  =\int_G \phi_0(x)f(x) dx$ for non-negative $f\in C_c(G)$. On the other hand, the existence of (completely bounded) characteristic tame cuts $(\varphi_n)$ suggests asymptotic information of how the operator norm is changed after truncating (or cutting off) a function $f\in C^*_\lambda (\Gamma)$ to the set $\supp (\varphi_n)$. In many spaces, the truncation process defines a norm decreasing operators. For instance, $\ell^p (G)$ for $1\leq p\leq \infty$ and more generally the space $C_c(\Gamma)$ endowed with an unconditional norm $N$, that is  $N(f) \leq N(g)$ whenever $|f|\leq |g|$. However, the situation is very different in $C^*_\lambda (\Gamma)$.
The simplest illustration can be seen in the case $\Gamma=\Z$. We identify $C^*_\lambda (\Z)$ and $C(\T)$  via Gelfand transform. Then the truncation operator associated to the characteristic function $\varphi_n = \n1_{[-n,n]\cap \Z}$ corresponds to the Dirichlet's kernel $D_n (t)= \sum_{j=-n}^n e^{i2\pi jt}$ in $L^1(\T)$ and it is well known that 
\[\|M_{\varphi_n}:{C^*_\lambda(\Z)\rightarrow C^*_\lambda(\Z)}\| = \|D_n\|_1 = \dfrac{4}{\pi}\log n+ O(1).\] This shows the operator norm of $C^*_\lambda (\Z)$ is not unconditional. In \cite{1981hardyineq}, it is shown that this logarithmic rate is asymptotically the best while the support of the truncation operator grows at least linearly. This is why we expect the truncation operators on $C^*_\lambda (\Gamma)$ to have bigger norms as its support grows, and we study characteristic tame cuts to understand their norm growth rate. In the following theorem, we provide some examples with (completely bounded) characteristic tame cuts.
\begin{thmD} The following groups have completely bounded characteristic tame cuts.
\begin{enumerate}[(i)]
\item $\Z^d\rtimes_A \Z$ for any $A\in SL(d,\Z)$.
\item $\Z [\frac{1}{pq}] \rtimes_{\frac{p}{q}} \Z$ for any coprime $p,q\in\N$.
\item Lamplighter groups $\Z_p \wr \Z$ for any $p\in \N$.
\item Baumslag-Solitar groups $BS(p,q)=\langle a,t\mid ta^pt^{-1}=a^q\rangle$ for any $p,q\in\N$.
\end{enumerate}
\end{thmD}

The paper is organized as follows. In Section \ref{sec defs}, we list out preliminary definitions and results. In Section \ref{PV}, we discuss how the tame cuts are connected to weak amenability and rapid decay property, and provide the first examples. In Section \ref{application}, we prove Theorem A and Theorem B. In Section \ref{PropertTSchur}, we prove Theorem C. In Section \ref{sec examples}, we prove Theorem D.

\section{Preliminaries}\label{sec defs}
\subsection{Group algebras}
Let $G$ be a locally compact group. We fix a Haar measure $dx$ on $G$. The left $G$-translations on itself induces \textit{the (left) regular representation} $\lambda: G\rightarrow \mathcal{U}(L^2(G))$. This gives rise to the *-representation
\begin{align*}
\lambda:L^1 (G) \rightarrow \B(L^2(G)), \quad \lambda(f)\xi = f*\xi.
\end{align*} The \textit{(left) reduced} $C^*$\textit{-algebra} $C^*_\lambda (G)$ of $G$ is the completion of $\lambda(L^1(G))$ in $\B(L^2(G))$ with respect to the operator norm topology. The \textit{group von Neumann algebra} $L(G)$ is the completion of $\lambda (L^1(G))$ in $\B (L^2(G))$ with respect to the weak* topology. Another important algebra is the \textit{Fourier algebra}, 
\begin{align*}
A(G) = \left\{\varphi = \sum_i \xi_i *\eta_i : \xi_i,\eta_i\in L^2(G), \sum_{i} \|\xi_i\|_2\|\eta_i\|_2<\infty\right\}.
\end{align*}
This is a commutative Banach algebra endowed with pointwise multiplication and the norm $\|\varphi\|_A = \inf \sum_{i} \|\xi_i\|_2\|\eta_i\|_2$, where $\xi_i$'s and $\eta_i$'s run in $L^2 (G)$ under the condition $\varphi = \sum_i \xi_i *\eta_i$. It turns out that the unique Banach predual of $L(G)$ is the Fourier algebra $A(G)$. The duality is given by $\langle \lambda(f) , \varphi\rangle = \int_G f(x)\varphi(x)dx$ for all $f\in L^1(G)$ and $\varphi\in A(G)$. 
We refer to Eymard's paper \cite{eymard1964algebre} for further properties  of these algebras. 

Let $\varphi$ be a function on $G$, and $F$ any set of functions on $G$. We say $\varphi$ multiplies $F$ if $\varphi F \subseteq F$. A function $\varphi$ on $G$ is called a \textit{Fourier multiplier } if it multiplies the Fourier algebra $A(G)$. By the closed graph theorem, the multiplication operator $m_\varphi:A(G)\rightarrow A(G)$ corresponding to a Fourier multiplier $\varphi$ is continuous, and the operator norm $\|\varphi\|_{MA} = \sup \{\|\varphi \psi\|_A: \psi\in A(G), \|\psi\|_A =1\}$ is well defined. It turns out that the space $MA(G)$ of all Fourier multipliers is a Banach algebra with respect to pointwise multiplication.

It is clear that any continuous function $\varphi$ multiplies $C_c(G)$. Denote by $m_{\varphi}:C_c(G)\rightarrow C_c(G)$ the corresponding multiplier map. Duality and restriction arguments show the followings are equivalent:
\begin{enumerate}[(i)]
\item $\varphi\in MA(G)$.
\item $m_\varphi$ extends to a weak*-continuous map $m_\varphi^* : L(G)\rightarrow L(G)$.
\item $m_\varphi$ extends to a norm-continuous map $m_\varphi': C^*_\lambda (G) \rightarrow C^*_\lambda (G)$.
\end{enumerate}
A Fourier multiplier $\varphi\in MA(G)$ is said \textit{completely bounded} if $m_\varphi^*$ (or equivalently $m_\varphi'$) is completely bounded, that is the quantity
\begin{align*}
\|\varphi\|_{M_0A} = \sup_{n\in\N} \|m_\varphi^* \otimes I_n :L(G) \otimes M_n(\C)\rightarrow L(G) \otimes M_n(\C)\|
\end{align*} is finite. 
Denote by $M_0A(G)$ all completely bounded Fourier multipliers. Again, it is a Banach algebra with respect to pointwise multiplication. Compared to Fourier multipliers, the characterizations of completely bounded Fourier multipliers are much more rich. Below, we list out some important results that we will use.
\begin{thm}[de Canniere-Haagerup, \cite{dCH1985}]
Let $\varphi$ be a Fourier multiplier of $G$. Put $K = SU(2)$. Then the following conditions are equivalent:
\begin{enumerate}[(i)]
\item $\varphi\in M_0A(G)$ with $\|\varphi\|_{M_0A(G)}\leq C$.
\item $\varphi\otimes \n1_K\in MA(G\times K)$ with $\|\varphi\otimes \n1_K\|_{MA(G\times K)}\leq C$.
\item For every locally compact group $H$, we have $\varphi\otimes \n1_H\in MA(G\times H)$ with $\|\varphi\otimes \n1_H\|_{MA(G\times H)}\leq C$.
\item There exist bounded continuous maps $P,Q$ from $G$ to a Hilbert space $\mathcal{H}$ with $\varphi (y^{-1}x) = \langle P(x),Q(y)\rangle$ for all $x,y\in G$ and $\sup_{x,y\in G} \|P(x)\| \|Q(y)\|\leq C$.
\end{enumerate}
\end{thm}
\begin{thm}[\cite{Nebbia1982,Losert1984}]
When $G$ is amenable, the Banach algebras $M_0A(G)$ and $MA(G)$ coincide, and the inclusion $A(G)\rightarrow MA(G)$ is isometric.
\end{thm}

\subsection{Weak amenability}
\begin{defn}
A locally compact group $G$ is \textit{weakly amenable} if there is a net $(\varphi_i)$ of compactly supported functions such that 
\begin{enumerate}[(i)]
\item $\varphi_i \rightarrow 1$ uniformly on all compact subsets of $G$,
\item There is a constant $C>0$ such that $\|\varphi_i\|_{M_0A} \leq C$.
\end{enumerate}
The best upper bound $\Lambda_G$ is called the \textit{Cowling-Haagerup constant}.
\end{defn}

Locally compact amenable groups are trivially weakly amenable. To mention some more weakly amenable groups: the free group $F_2$, more generally Gromov's hyperbolic groups, groups acting properly on a finite dimentional CAT(0) cubical complexes, and rank one simple Lie groups. The main importance of weak amenability is that the Cowling-Haagerup constant is measure equivalence invariant and W*-equivalence invariant. We refer to \cite{haagerup1989,haagerup2016,dCH1985} for more on weak amenability.

In the following, we mention two results that we will use.
\begin{lem}[Cowling-Haagerup,\cite{haagerup1989}]\label{KKavg}
Suppose $K$ is a compact subgroup of $G$. Let $\varphi\in M_0A(G)$ (resp. $\varphi\in MA(G)$) and $\dot{\varphi}$ denotes the function obtained by averaging $\varphi$ over the double cosets $KxK, (x\in G)$, that is 
\begin{align*}
\dot{\varphi} (x) = \int_K\int_K\varphi(k_1x k_2)dk_1dk_2 \quad \text{for all}\quad  x\in G
\end{align*} where $dk_1$ and $dk_2$ are normalized Haar measures on $K$.  Then $\dot{\varphi}\in M_0A(G)$ and $\|\dot{\varphi}\|_{M_0A}\leq \|\varphi\|_{M_0A}$ (resp. $\|\dot{\varphi}\|_{MA}\leq \|\varphi\|_{MA}$). Furthermore, if $G$ admits a closed amenable subgroup $P$ such that $G=PK$, then 
\begin{align*}
\|\dot{\varphi}\|_{M_0A(G)}=\|\dot{\varphi}\|_{MA(G)}=\|\dot{\varphi}|_P\|_{MA(P)}.
\end{align*} 
\end{lem}

\begin{thm}[Lafforgue-de la Salle, \cite{lafforgue2011CBAP}]\label{rigidity ineq thm}
Let $G=SL(3,\R)$, and $K=SO(3,\R)$. There is a constant $C> 0$ such that for any $K$-bi-invariant function $\varphi\in C_0 (G)$ and any $t> 0$, we have
\begin{align}\label{rigidity ineq}
C e^{t/3}  \left| x{\varphi} \begin{pmatrix}
e^t & 0&0\\
0 &1&0\\
0 &0& e^{-t}
\end{pmatrix}\right| \leq \|\varphi\|_{MA}.
\end{align}
\end{thm}
The same result was proved for $Sp(2,\R)$ in \cite{deLaat2012}. Combined together, these two inequalities show that the higher rank simple Lie groups with finite center are not weakly amenable.

\subsection{Rapid decay property (RD)}\label{rapid decay}
First observed by Uffe Haagerup in \cite{haagerupanexample} for free groups, and later developed by Paul Jolissaint in \cite{jolissaint1990rapidly}, the Rapid decay property took significant attention of mathematicians for being a part of tools to establish important examples of groups giving a positive answer to Baum-Connes conjecture. The class of groups with Rapid decay property is quite large and includes Gromov's hyperbolic groups, co-compact cubical CAT(0) groups, Coxeter groups, Mapping class groups, Braid groups, large type Artin groups, 3-manifold groups not containing $Sol$, Wise non-Hopfian group, some small cancellation groups, and uniform lattices of $SL(3,\mathbb{R})$. We invite the readers to \cite{RDintro} for a comprehensive survey.

Recall that a \textit{length function} $\ell: G\rightarrow \R_+$ on a locally compact group $G$ is a measurable function such that $\ell(x) = \ell(x^{-1})$ and  $\ell(xy)\leq \ell (x)+\ell (y)$ for all $x,y\in G$, and $\ell(e) =0$ where $e\in G$ is the neutral element. The ball of radius $n\geq 0$ with respect to $\ell$ is the set $B_n = \{x\in G: \ell (x)\leq n\}$. We say a length function $\ell_1$ dominates another length function $\ell_2$ if there is a constant $c>0$ such that $\ell_2(x)\leq c\ell_1(x) +c$ for all $x\in G$. If two length functions dominate each other, we say they are equivalent. Let us provide an interesting example of equivalent length function. Consider the group $G=SL(3,\R)$. Take any compact generating set $C\subseteq G$. The word length function $\ell_C$ associated to $C$ is defined by $
\ell_C (x) = \min \{n\in\Z_+:x\in C^n\}$, and it dominates any other length functions on $G$. Another length function on $G$ can be defined as $L(x) = \log \|x\|+\log\|x^{-1}\|$ where $\|x\|$ is the operator norm of $x$ acting on $\ell^2(\{1,2,3\})$. The latter length function  comes from the Cartan decomposition. It turns out these two norms are equivalent. 
Let $\Gamma$ be any uniform lattice in $G$. Since $\Gamma$ satisfies the Kazhdan's property (T), there is a finite generating set $S\subseteq \Gamma$. The length functions $\ell_S$, $\ell_C|_S$, and $L|_S$ on $\Gamma$ are all equivalent on $\Gamma$ (see \cite{LMR93,LMR00}).

\begin{defn}
A locally compact group $G$ has \textit{rapid decay property} (RD in short) with respect to a length function $\ell$ if there exist non-negative constants $C$ and $a$ such that for all $n\geq 0$ and $f\in C_c(G)$ with $\supp (f) \subseteq B_n$ we have
\begin{align*}
\|\lambda(f)\| \leq C(1+n)^a \|f\|_2.
\end{align*}
\end{defn}
We highlight one of Lafforgue's results, which was the last step to prove that uniform lattices in $SL(3,\R)$ satisfy the Baum-Connes' conjecture. 
\begin{thm}[Lafforgue, \cite{LafforgueRD}]\label{Vincent1}
Uniform lattices in $SL(3,\R)$ have RD.
\end{thm}
This theorem is the main ingredient to prove Theorem A.
\subsection{Property $(T_{Schur},G,K)$}\label{PropertTSchur}
Property $(T_{Schur}, G, K)$ was introduced in \cite{Lia16} as an analogue of property $(T_{Schur})$. Suppose that $G$ is a reductive group over a local field, $K$ is its maximal compact subgroup, and $\Gamma$ is a lattice of $G$ satisfying property $(T_{Schur},G,K)$. Then, all known methods to prove the Baum-Connes conjecture fail for $\Gamma$. 

\begin{defn}\label{defn propert T schur} Let $G$ be a locally compact group, $K$ a compact subgroup, $H$ a closed subgroup of $G$, and $\ell$ a proper length function of $G$. For $n\in \N$ and $f\in C(G)$, define the quantity
\begin{align*}
\|f\|_{MA(H,G,\ell,n)} = \sup \left\{ \|\lambda_H(f|_H \varphi)\| : \varphi\in C_c(H),\supp(\varphi)\subseteq B_n, \|\lambda_H(\varphi)\|\leq 1 \right\}.
\end{align*}
When $G$ and $\ell$ are already fixed, we also write $\|f\|_{MA(H,G,\ell,n)} =\|f\|_{MA(H,n)}$. We say that $H$ has \textit{property} $(T_{Schur},G,K,\ell)$ if there exist a positive constant $s > 0$ and a function $\phi\in C_0(G)$ vanishing at infinity such that for any $D > 0$ and $K$-bi-invariant
function $\varphi\in C(G)$ with the following  condition 
\begin{align*}
\|\varphi\|_{MA(H,n)}\leq De^{sn}\quad \text{for all}\quad n\in\N,
\end{align*}
there exists a limit $\varphi_\infty\in\C$ to which $\varphi$ tends uniformly rapidly
\begin{align*}
|\varphi(x)-\varphi_\infty|\leq D\phi (x)\quad \text{for all} \quad x\in G.
\end{align*}
\end{defn}

The only known example given in \cite{Lia16} is as follows. Let $\mathbb{F}_q$ be a finite field of characteristic different from 2 with cardinality $q$.  Let $G$ be the symplectic group $Sp_4(\mathbb{F}_q((\pi)))$ over the local field $\mathbb{F}_q ((\pi))$ and $K =Sp_4(\mathbb{F}_q [[\pi]])$ the maximal compact subgroup of $G$. Let $\Gamma$ be the non-uniform lattice $Sp_4 (\mathbb{F}_q[\pi^{-1}])$ in $G$. Let $H<\Gamma$ be the subgroup consisting of the elements of the form 
\begin{align*}
\begin{pmatrix}
1&*&*&*\\&1&*&*\\&&1&*\\&&&1
\end{pmatrix}\in \Gamma.
\end{align*} 
For $i,j\in \N_0$, denote $D(i,j) = diag(
\pi^{-i}, \pi^{-j},\pi^j,\pi^i)$. By Cartan's decomposition theorem, every element $x\in G$ can be written as $x=kD(i,j)k'$ for some $k,k'\in K$ and a unique $(i,j)\in\N_0^2$ with $i\geq j$. Moreover the length function $\ell: kD(i,j)k'\mapsto i+j$ is equivalent to the word length function of $G$, and even its restriction to the lattice $\Gamma$ is equivalent to the word length function of $\Gamma$. Then $H$ and $\Gamma$ have property $(T_{Schur},G,K,\ell)$.

It seems that Kazhdan's property (T) has less to do with property $(T_{Schur},G,K,\ell)$. For example, compact groups obviously satisfy property (T), while it is not the case for property $(T_{Schur},G,K,\ell)$ as the following lemma shows.
\begin{lem}
Let $G$ be an unbounded locally compact group endowed with a proper length function $\ell$. Suppose that $K$ and $H$ are compact subgroups of $G$. Then $H$ does not have property $(T_{Schur}, G, K, \ell)$.
\end{lem}

\begin{proof}
Assume, by contradiction, that $H$ has property $(T_{Schur}, G, K, \ell)$ and let $s>0$ and $\phi\in C_0(G)$ be as in Definition \ref{defn propert T schur}.
Define
\begin{align*}
L_1(x) &= \left\lbrace\begin{array}{ll}
0, &\text{if } \ell(x) = 0\\
\lfloor \ell(x)\rfloor +1, &\text{otherwise}
\end{array}\right.\\
L_2(x) &=\int_K L_1(kxk^{-1})dk\\
L_3 (x) &=\min \{L_2 (k_1xk_2):k_1,k_2\in K\}\\
L_4 (x) &= \left\lbrace\begin{array}{ll}
0, & \text{if } x\in K\\
1, & \text{otherwise.}
\end{array}\right.
\end{align*} for all $x\in G$. This construction is from \cite{jolissaint1990rapidly}. Note that $L_3+L_4$ is a length function equivalent to $\ell$. If necessary, replacing $\ell$ by $L_3+L_4$, we can assume that $\ell$ takes 0 on $K$ so that the balls are $K$-bi-invariant. Choose a large enough $r\in\N$ such that $B_r$ has a non-empty interior and contains $H$. For each $m\geq r$, construct a nonnegative function $f_m\in A(G)$ such that $f_m$ takes 1 on $B_m$ and 0 on $B_{2m}^0$ using Eymard's trick. Then $\psi_m = f_{4m} - f_m$ is a non-negative compactly supported  function in $A(G)$ such that $\psi_m |_H = 0$ and $\psi_m (x) \neq 0$ for some $x\in G$ with $\ell (x) \geq m$. We use $K\times K$ double averaging and normalization on $\psi_m$ in order to have a $K$-bi-invariant function ${\varphi}_m\in C_c(G)$ such that $\varphi_m|_H = 0$ and ${\varphi}_m (x) = 1$ for some $x\in G$ with $\ell (x) \geq m$. Now, we have
$\|\varphi_m\|_{MA(H,n)} = 0 \leq e^{sn}$ for all $n\in \N$ and $m\geq r$. It follows that $| \varphi_m(x)| \leq \phi(x)$ for all $x\in G$  and $m\geq r$. Taking $\lim_{x\rightarrow \infty} \sup_{m\geq r}$ on left and right hand side of the inequality, we get a desired contradiction.
\end{proof}

In \cite[Proposition 2.3]{Lia16}, it was proved that if a discrete subgroup $\Gamma$ of $G$ has property $(T_{Schur},G,K,\ell)$, then $(\Gamma,\ell)$ does not have RD. As its analogue, we have Theorem D.

\section{Groups with tame cuts}\label{PV}
In this section, we will give the first examples of groups with tame cuts by investigating its connection to weak amenability, multiplier approximation property, and RD. If there exists a proper length function $\ell$ for which $(G,\ell)$ has (completely bounded) [characteristic] tame cuts, we simply say $G$ has (completely bounded) [characteristic] tame cuts. It is also worth noting that when $G$ is amenable, $G$ has (characteristic) tame cuts if and only if $G$ has completely bounded (characteristic) tame cuts since in that case, the Banach algebras $M_0A(G)$ and $MA(G)$ coincide. Here are the first examples.
\begin{prop}\label{example Z}
The infinite cyclic group $\Gamma = \Z$ has characteristic tame cuts with respect to the logarithmic length function $\log (1+|\cdot |)$ but does not with respect to the double logarithmic length function $\log (1+\log (1+|\cdot |))$.
\end{prop}
\begin{proof}
We use the Fourier transform and bounds for $L^1(\T)$-norm of trigonometric polynomials.
\begin{enumerate}[(i)]
\item For $n\in \N$, we have
\begin{align}\label{hardyeq}
\begin{split}
\|\n1_{\Z \cap [-n,n]}\|_{M_0A} &= \|\n1_{\Z \cap [-n,n]}\|_A = \|\mathcal{F}(\n1_{\Z \cap [-n,n]})\|_{L^1(\T)} \\&= \left\|\sum_{k=-n}^n e^{i2\pi kt}\right\|_{L^1(\T)} = \dfrac{4}{\pi} \log (n) + O(1).
\end{split}
\end{align} (See \cite[p. 71, Exercise 1.1]{MR2039503}). This shows that the sequence $(\n1_{3^n})_{n\in\N}$ forms characteristic tame cuts for $(\Z,\log(1+|\cdot|))$.
\item  There is a constant $C>0$ such that for any finite subset $B\subseteq \Z$ we have
\begin{align}\label{hardyineqrigid}
\|\n1_B\|_{M_0A} = \left\|\sum_{k\in B} e^{i2\pi kt}\right\|_{L^1 (\T)} \geq C\log |B|.
\end{align} (See \cite{1981hardyineq}). If $\varphi_n$ is a characteristic function on $\Z$ such that its support is finite and contains the ball $B_n$ of radius $n$, then the set $\supp (\varphi_n)$ would contain at least $e^{e^n} \approx |B_n|$ elements. The inequality $(\ref{hardyineqrigid})$ shows that $\|\varphi_n\|_{M_0A}\geq C e^n$. Thus, there is no characteristic tame cuts for $(\Z, \log (1+\log (1+|\cdot |)))$.
\end{enumerate}
\end{proof}
\begin{rmk}
The sequence $(\n1_{\Z \cap [-n,n]})$ gives characteristic tame cuts for $(\Z, \log (1+|\cdot |))$, and the multiplier norms grow linearly. The inequality $(\ref{hardyineqrigid})$ shows that this linear rate is actually the best.
\end{rmk}
\begin{prop}\label{example weak amenability}
Let $G$ be a locally compact group, and $\ell$ be a proper length function on it. Then $(\Gamma,\ell)$ has completely bounded tame cuts with $a=0$ (cf. Definition \ref{def1}) if and only if $G$ is weakly amenable.
\end{prop}
\begin{proof} See \cite[Proposition 1.1]{haagerup1989}.
\end{proof}
\begin{rmk} Similar arguments can be made for other approximation properties such as  multiplier approximation property and $n$-positive approximation property.
\end{rmk}
\begin{rmk} By Proposition \ref{example Z} and Proposition \ref{example weak amenability}, the group  $(\Z, \log(1+\log (1+|\cdot|)))$ has tame cuts but no characteristic tame cuts. 
\end{rmk}
\begin{prop}
Let $\Gamma$ be a discrete group satisfying RD with respect to a proper length function $\ell$. Then $(\Gamma,\ell)$ has characteristic tame cuts.
\end{prop}
\begin{proof}
Note that for any non-zero function $f\in C_c(\Gamma)$, we have
\begin{align*}
\|\lambda(f)\| \geq \dfrac{\langle \lambda (f)\delta_e| f\rangle }{\|f\|_2} = \|f\|_2.
\end{align*} It follows that if we choose $\varphi_n = \n1_{B_n}$, we have 
\begin{align*}
\begin{split}
\|\lambda(\varphi_n f)\| &\leq C(1+n)^a \|\varphi_n f\|_2 \\
&\leq C(1+n)^a \|f\|_2 \\
&\leq C(1+n)^a \|\lambda (f)\|,
\end{split}
\end{align*} 
for all $f\in C_c(\Gamma)$. It follows that $\|\varphi_n\|_{MA}\leq C(1+n)^a$, completing the proof.
\end{proof}

\section{Proof of Theorem A and Theorem B}\label{application}
As a motivation to investigate tame cuts, we provide two applications of tame cuts. The first one is about to continuity of the induction map. The following lemma shows the importance of induction map and is used to prove that amenability, a-T-menability, weak-amenability, and Kazhdan's property (T) are inherited by lattices.
\begin{lem}[Haagerup, \cite{haagerup2016}]\label{induction lemma} Let $\Gamma$ be a lattice in a locally compact group $G$. Let $\sigma:G/\Gamma\rightarrow G$ be a Borel cross section, and $\Omega\subseteq G$ its image. Denote by $\mu_\Gamma$ the counting measure on $G$ corresponding to $\Gamma$. Then
\begin{align*}
\Phi:\ell^\infty (\Gamma) \rightarrow C_b (G),\quad \varphi\mapsto\widehat{\varphi} = \n1_{\Omega} * (\varphi \mu_\Gamma)*\widetilde{\n1}_{\Omega}
\end{align*} is well defined. Moreover, the restrictions of this map give bounded linear maps 
\begin{align*}
\Phi: A(\Gamma)\rightarrow A(G)\quad and \quad \Phi:M_0A(\Gamma)\rightarrow M_0A(G).
\end{align*}
Moreover, if $\varphi$ is a positive definite function on $\Gamma$, so is $\widehat{\varphi}$ on $G$.
\end{lem}
As we have a disjoint union $G=\bigsqcup_{y\in \Gamma} \Omega y$, the map $\gamma : G\rightarrow \Gamma$ such that $x\in \Omega \gamma (x)$ is well defined, and we have 
\begin{align*}
\Phi (x) = \int_{\Omega} \varphi(\gamma (x\omega)) d\omega\quad \text{for all} \quad x\in G.
\end{align*} where $d\omega$ is the normalized Haar measure on $G$ so that $\int_\Omega d\omega = 1$.

\begin{thmA}\label{thmA}
Let $\Gamma$ be a uniform  lattice in $G=SL(3,\R)$. Then the map
\begin{align*}
\Phi:MA(\Gamma)\rightarrow MA(G), \quad \varphi\mapsto\widehat{\varphi} = \n1_{\Omega} * (\varphi \mu_\Gamma)*\widetilde{\n1}_{\Omega}
\end{align*} is not bounded.
\end{thmA}
\begin{proof} Recall that the length functions $\ell_S$, $\ell_C|_\Gamma$, and $L|_\Gamma$ are all equivalent on $\Gamma$ (cf. Section \ref{rapid decay}). We choose $L$ as the main length function on both $\Gamma$ and $G$.

By contradiction, suppose the map 
$\Phi:MA(\Gamma)\rightarrow MA(G)$ is bounded. By Theorem \ref{Vincent1}, $\Gamma$ has RD, and a fortiori characteristic tame cuts, so there is finitely supported characteristic functions $\varphi_n$ such that 
\begin{align*}
\|\varphi_n\|_{MA(\Gamma)}\leq Cn^a \quad \text{and} \quad \varphi_n |_{B_n} = 1  \quad \text{for all} \quad n\in\N.
\end{align*} By continuity of $\Phi$, there is a constant $C'>0$ such that  
\begin{align*}
\|\widehat{\varphi}_n\|_{MA(G)}\leq C' n^a \quad \text{for all} \quad n\in\N.
\end{align*}
The $K$-$K$-averaging of $\widehat{\varphi}_n$  is a $K$-bi-invariant Fourier multiplier of $G$ with 
\begin{align*}
\|\dot{\widehat{\varphi}}_n\|_{MA(G)} \leq \|\widehat{\varphi}_n\|_{MA(G)}\leq C'n^a \quad \text{for all} \quad n\in\N.
\end{align*}
Applying the rigidity inequality (\ref{rigidity ineq}) on $\dot{\widehat{\varphi}}_n$, we get 
\begin{align}\label{ineq1}
C'' e^{t/3}  \left| \dot{\widehat{\varphi}}_n \begin{pmatrix}
e^t & 0&0\\
0 &1&0\\
0 &0& e^{-t}
\end{pmatrix}\right| \leq \|\dot{\widehat{\varphi}}_n\|_{MA}\leq C'n^a \quad \text{for all}\quad n\in\N, t\in\R_+.
\end{align} Put $c=\max \{L (\omega):\omega \in \Omega\}$ and 
choose $t = n/4 - c$ so that $L(diag(e^t,1,e^{-t})) = 2t <n -2c$. Recall that we have
\begin{align*}
\dot{\widehat{\varphi}}_n (x) = 
\int_K \int_K \int_\Omega \varphi_n (\gamma (k_1 x k_2 \omega)) d\omega dk_1 dk_2.
\end{align*} by construction.
Note that 
\begin{align*}
L(\gamma (k_1 x k_2 \omega )) = L(\omega 'k_1 x k_2\omega )\leq L(k_1 xk_2) +L(\omega ') +L(\omega)\leq 2c + L(x).
\end{align*} Thus, if $L(x)\leq n-2c$, we have $\dot{\widehat{\varphi}}_n (x) = 1$. Therefore, the left hand side of (\ref{ineq1}) grows exponentially while the right hand side is polynomial, which gives the desired contradiction.
\end{proof}
\begin{rmk}
We do not know if the map $\Phi : MA(\Gamma)\rightarrow MA(G)$ is well defined.
\end{rmk}

\begin{thmB}\label{prop17}
Let $\Gamma$ be a uniform lattice in $G=SL(3,\R)$. Then $M_0 A(\Gamma)$ is a proper subalgebra of $MA(\Gamma)$.
\end{thmB}
\begin{proof}
Suppose $M_0A(\Gamma) = MA(\Gamma)$. Since the inclusion $M_0A(\Gamma)\rightarrow MA(\Gamma)$ is a contraction between two Banach spaces, two norms $\|\cdot\|_{MA}$ and $\|\cdot\|_{M_0A}$ are equivalent by the closed graph theorem applied on the inverse map.

Let us use the functions $\varphi_n$, $\widehat{\varphi}_n$, and $\dot{\widehat{\varphi}}_n$ from Theorem A. Since two norms are equivalent and $(\varphi_n)$ is tame cuts, this sequence is also completely bounded tame cuts. It follows that there are constants $C,a\geq 0$ such that
\begin{align*}
\|\varphi_n\|_{M_0A}\leq Cn^a \quad \text{for all} \quad n\in\N.
\end{align*}
By Lemma \ref{induction lemma}, we have $\|\dot{\widehat{\varphi}}_n\|_{M_0A}\leq \|\widehat{\varphi}_n\|_{M_0A}\leq Cn^a$. Again the rigidity inequality on $\dot{\widehat{\varphi}}_n$ gives a desired contradiction for the same choice of $t$ as in Theorem A.
\end{proof}

\section{Proof of Theorem C}
\begin{defn}
Let $G$ be a locally compact group, $H$ a closed subgroup of $G$, and $K$ a compact subgroup of $G$. A function $\varphi \in C(H)$ is said \textit{K-bi-invariant} if
there is a $K$-bi-invariant continuous function on $G$ whose restriction on $H$ is exactly $\varphi$.
\end{defn}
\begin{thmC}\label{ppt T schur no tame cuts}
Let $H$ be closed subgroup of an unbounded locally compact group $G$. Suppose that $H$ satisfies property $(T_{Schur},G,K,\ell)$ for a compact subgroup $K$ and a proper length function $\ell$ of $G$. Then $(H,\ell|_H)$ does not admit any $K$-bi-invariant tame cuts.
\end{thmC}

\begin{proof}
We prove by contradiction. Assume that there exists $K$-bi-invariant tame cuts $(\varphi_m)_{m\in\N}$ for $(H,\ell|_H)$. There are constants $C,a\geq 0$ such that
\begin{align}
\label{363310}&\|\varphi_m\|_{MA(H)}\leq C m^{a}\\
\label{363311}&\varphi_m|_{B_m}\equiv 1 
\end{align} for all $m\in\N$. Take any function $f\in C_c(H)$ with $\supp (f) \subseteq B_n$. If $m\geq n$, the Fourier multiplier $\varphi_m$ acts trivially on $f$, so $\|\lambda_H \left (\varphi_m f\right )\|= \|\lambda_H(f)\|$. If $m<n$, then we have 
\begin{align*}
\|\lambda_H \left (\varphi_m f\right )\|\leq Cm^a \|\lambda_H(f)\|\leq Cn^a \|\lambda_H(f)\|
\end{align*} by $(\ref{363310})$.
Unifying these two cases, if we denote by $\varphi_m'\in C(G)$ a $K$-bi-invariant extension of $\varphi_m$, we get the inequality $\|\varphi_m'\|_{MA(H,n)}\leq C n^a$ for all $n,m\in \N$. Let $s>0$ and $\phi\in C_0(G)$ be from property $(T_{Schur},G,K,\ell)$. Put $D = \sup_{n\in\N} Cn^a e^{-sn}$
so that we get 
\begin{align*}
\|\varphi_m'\|_{MA(H,n)}\leq D e^{sn}\quad \text{for all}\quad m,n\in \N.
\end{align*} By property $(T_{Schur},G,K,\ell)$, we get 
\begin{align}\label{T_Schur contradicting ineq}
|\varphi_m(x)|=|\varphi_m' (x)|\leq D\phi (x) \quad \text{for all} \quad x\in H, m\in \N.
\end{align} Now, if we take the sequential limits $\lim_{x\rightarrow\infty}\lim_{m\rightarrow\infty}$ on $(\ref{T_Schur contradicting ineq})$, the left hand side goes to 1 whereas the right hand side goes to 0. This gives a desired contradiction.
\end{proof}
\begin{cor}
If $G$ is unbounded and has property $(T_{Schur},G,K,\ell)$, then $(G,\ell)$ does not have tame cuts.
\end{cor}
\begin{proof}
$K$-bi-averaging process does not increase multiplier norm by Lemma \ref{KKavg}. Thus, we can assume the functions in tame cuts are all $K$-bi-invariant. Now the rest follows from the theorem.
\end{proof}
\begin{cor} Suppose that $G$ is a finitely generated infinite group and $H$ is a finitely generated subgroup of $G$. Suppose that $H$ is at most polynomially distorted in $G$. Recall that $H$ is polynomially distorted in $G$ if there exists $k\geq 0$ such that $\ell_H(x)\leq k\ell_G (x)^k + k$ for all $x\in H$, where $\ell_G$ and $\ell_H$ are the word length functions of $G$ and $H$, respectively. If $H$ has property $(T_{Schur}, G,\{e\},\ell_G)$, then $(H,\ell_H)$ does not have tame cuts.
\end{cor}

We do not have any example of groups fitting to the corollaries above.

\section{Proof of Theorem D}\label{sec examples}
In this section, we prove Theorem D case by case. That is achieved by exploiting the following stability properties.
\begin{prop}\label{restriction map}
Restrictions of (completely bounded) [characteristic] tame cuts on a closed subgroup are again (completely bounded) [characteristic] tame cuts.
\end{prop}
\begin{proof}
Directly follows from \cite[Proposition 1.12]{dCH1985} which states
 if $G$ is a locally compact group and $H$ is its closed subgroup, then 
 the restriction maps
\begin{align*}
MA(G)\rightarrow MA(H)\quad and \quad M_0A(G) \rightarrow M_0 A(H)
\end{align*}
 are norm decreasing.
\end{proof}
\begin{lem}\label{extension from open subgroup}
Let $G$ be a locally compact group, and $H$ an open subgroup. The extensions by identity on $H$ and zero on $G\setminus H$ 
\begin{align*}
A(H)\rightarrow A(G),\quad MA(H)\rightarrow MA(G),\quad and \quad M_0A(H)\rightarrow M_0A(G)
\end{align*} are isometric.
\end{lem}
\begin{proof}
The case of Fourier algebra is proven in \cite{eymard1964algebre}.
Denote by $\iota : MA(H)\rightarrow MA(G)$ the trivial extension. 
\begin{align*}
\|\iota (\varphi)\|_{MA(G)} &= \sup \{ \|\iota(\varphi)\psi\|_{A(G)}: \psi \in  A(G), \|\psi\|_{A(G)}\leq 1\}\\ &=  \sup \{ \|\varphi \psi|_H\|_{A(H)}: \psi \in  A(G), \|\psi\|_{A(G)}\leq 1\}\\
& = \sup \{ \|\varphi \psi'\|_{A(H)}: \psi' \in  A(H), \|\psi'\|_{A(H)}\leq 1\} = \|\varphi\|_{MA(H)},
\end{align*} hence $\iota $ is isometric.
The case $M_0A(H) \rightarrow M_0A(G)$ is done similarly by considering $H\times SU(2)$ and $ G\times SU(2)$.
\end{proof}
\begin{prop}\label{polynomialext}
Let $\Gamma$ be a discrete group, $\ell:\Gamma\rightarrow \Z_+$ a proper length function, and $H$ a subgroup of $\Gamma$. Assume that $(H,\ell|_H)$ has (completely bounded) [characteristic] tame cuts and that $H$ has polynomial co-growth. Then $\Gamma$ has (completely bounded) [characteristic] tame cuts.
\end{prop}
Here, we say $H$ has polynomial co-growth with respect to $\ell$ if there exist constants $C',b\geq 0$ such that 
\begin{align*}
\# \left\{ xH : x\in \Gamma, \ell(x)\leq n\right\} \leq C'n^{b} \quad \text{for all}\quad n\in\N.
\end{align*} When $H$ is normal in $G$, it is equivalent to say the quotient group $G/H$ has polynomial growth with respect to the length function given by $\ell^H (xH) =\min \{ \ell (xh):h\in H\}$.
\begin{proof}[Proof of Proposition \ref{polynomialext}]
By hypothesis, there exist constants $C,a\geq 0$ and a sequence of finitely supported functions $(\psi_n)$ on $H$ with 
\begin{align}\label{BncapH}
 \psi_n|_{B_n\cap H} \equiv 1 \quad and \quad \|\psi_n\|_{MA(H)}\leq Cn^a \quad \text{for all}\quad n\in\N.
\end{align} We identify $\psi_n$ with the function on $\Gamma$ extending $\psi_n$ by 0 on $H^c$. By Lemma \ref{extension from open subgroup},
\begin{align*}
\|\psi_n\|_{MA(H)} = \|\psi_n\|_{MA(\Gamma)}.
\end{align*}
Choose a cross section $\sigma : \Gamma /H\rightarrow \Gamma$ in a way that $\ell(\sigma (xH)) = \min_{h\in H} \ell (xh)$. We can choose such $\sigma$ since the length function  $\ell$ takes value in $\Z_+$. Then by polynomial co-growth, the sets
\begin{align*} 
S_n = \{y \in \sigma(\Gamma/H): \ell(y) \leq n \}
\end{align*} grow at polynomial rate, say $C'n^b$. Also, note that \begin{align}\label{Bnsubset}
B_n \subseteq A_n (B_{2n} \cap H) \quad \text{for all}\quad n\in\N.
\end{align} To see that, take any $x\in B_n$.  By the choice of the cross section, we have $\sigma (xH)\in S_n\subseteq B_n$, and since $x$ and $\sigma(x)$ represent the same class, we have $\sigma(x)^{-1}x\in B_{2n}\cap H$. Now $x\in \sigma(x) (B_{2n}\cap H) \subseteq S_n (B_{2n}\cap H)$.

For every $n\in\N$, define a finitely supported function by $\varphi_n = \sum_{y\in S_n} \delta_y * \psi_{2n}$ on $\Gamma$. Equation (\ref{Bnsubset}), together with (\ref{BncapH}), shows that $\supp(\varphi_n)|_{B_n} \equiv 1$. Since $\Gamma$ acts on $MA(\Gamma)$ by isometries, we have
\begin{align*}
\|\varphi_n\|_{MA(\Gamma)} \leq \sum_{y\in S_n} \|\delta_y*\psi_{2n}\|_{MA(\Gamma)} \leq C''n^b \|\psi_{2n}\|_{MA(\Gamma)}\leq CC''2^a n^{a+b}.
\end{align*}
This completes the proof for tame cuts. The same proof works for the other cases.
\end{proof} 
\begin{prop}
Let $F$ be a finite group, and $P$ be a finitely generated group with polynomial growth. Then the wreath product $F\wr P$ has characteristic tame cuts. In particular, the Lamplighter group $\Gamma =\Z_2\wr\Z$ has characteristic tame cuts.
\end{prop}
\begin{proof}
Recall that $F\wr P$ is the semidirect product group $(\oplus_{n\in P} F)\rtimes_\rho P$ where $\rho$ acts by shifts. We choose a word length function  $\ell$ on $\Gamma$ associated to any finite generating set. By Proposition \ref{polynomialext}, it is enough to show that  $(H,\ell|_H)$ has characteristic tame cuts, where $H = \oplus_{n\in P} F$. 

For $n\in\N$, define $H_n$ the subgroup of $H$ generated by $B_n\cap H$. Since $H$ is locally finite, $H_n$ is a finite group. Define $\varphi_n = \n1_{H_n}$. By Lemma \ref{extension from open subgroup}, we have
$\|\varphi_n\|_{M_0A(H)} = \|\varphi_n\|_{M_0A(H_n)} = 1$
because the multiplier $\varphi_n = \n1_{H_n}$ acts on $A(H_n)$ trivially. The support of $\varphi_n$ is clearly finite and contains the relative ball $\{x\in H: \ell (x)\leq n\}$ of radius $n$.
\end{proof}
\begin{prop}\label{examsemidir}  For any coprime $p,q\in\N$, the group $\Gamma= \Z [\frac{1}{pq}] \rtimes_{\frac{p}{q}} \Z$ has characteristic tame cuts.
\end{prop}
\begin{proof}
Recall $\Gamma= \Z [\frac{1}{pq}] \rtimes_{\frac{p}{q}} \Z$ is isomorphic to the subgroup
\begin{align*}
\Gamma\cong\left\{ \left( \begin{matrix}
\left(\frac{p}{q}\right)^k & P\\ 0 & 1
\end{matrix}\right) :  k\in \Z , P\in \Z [\frac{1}{pq} ]\right\}
\end{align*}  of $GL_2 (\R)$
generated by the set
\[S = \left\{ s=\left( \begin{matrix}
\frac{p}{q} & 0\\ 0 & 1
\end{matrix} \right) , t  = \left( \begin{matrix}
1 & 1\\ 0 & 1
\end{matrix} \right) \right\}.\] Let $\ell$ be the associated word length  function on $\Gamma$. We only need to prove that  the subgroup $H\cong\left\{ \left( \begin{matrix}
1 & P\\ 0 & 1
\end{matrix}\right) : P\in \Z [\frac{1}{pq} ]\right\} \cong \Z [\frac{1}{pq}]$ has characteristic tame cuts with respect to the restricted length function $\ell |_H$.

Suppose $x = \left( \begin{matrix}
1 & P\\ 0 & 1
\end{matrix} \right) \in H$ and  $\ell (x)\leq n$. Then there is a decomposition
\[x = s_1...s_n\] with $s_i=\left( \begin{matrix}
\left(\frac{p}{q} \right)^{\varepsilon_i} & \delta_i \\ 0 & 1
\end{matrix} \right) \in S^{\pm 1}$, $\varepsilon_i,\delta_i \in \{-1,0,1\}$, $1\leq i\leq n$ and $\sum \varepsilon_i=0$.  Moreover, we have 
\begin{align*}
P &= \delta_1 + \delta_2 \left( \frac{p}{q} \right)^{\varepsilon_{1}} + \delta_3 \left( \frac{p}{q} \right)^{\varepsilon_{1} + \varepsilon_2} +... + \delta_n \left( \frac{p}{q} \right)^{\varepsilon_{1}+ ...+\varepsilon_{n-1}} 
\\&= \dfrac{\sum\limits_{i=1}^n\delta_i p^{n+\sum_{j=1}^{i-1}\varepsilon_j}q^{n-\sum_{j=1}^{i-1}\varepsilon_j} }{q^n p^n}.
\end{align*}
From this, it is easy to see that the cyclic subgroup $H_n$ of $H$ generated by the element \[ x_n = \left( \begin{matrix}
1 & \frac{1}{q^{n}p^n}\\ 0 & 1
\end{matrix} \right) \in H\]  contains the relative ball $B_n\cap H = \{x\in H: \ell (x) \leq n\}$. Moreover, for any element $x\in B_n\cap H$, its power in $H_n=\langle x_n \rangle$ has an upper bound  
\begin{align*}
|x|&=  \left|\sum\limits_{i=1}^n\delta_i p^{n+\sum_{j=1}^{i-1} \varepsilon_j}q^{n-\sum_{j=1}^{i-1}\varepsilon_j} \right|\leq n q^{2n} p^{2n}.
\end{align*} Denote by $A_n$ the subset of $H_n$ containing all elements with absolute power less than $nq^{2n}p^{2n}$. Note that $A_n$ is finite set containing $B_n\cap H$. Combining the facts $H$ is amenable, Lemma \ref{extension from open subgroup}, and (\ref{hardyeq}), we get
\begin{align}\label{amenable,subgroup}
\begin{split}
\|\n1_{A_n}\|_{M_0A(H)} &= \|\n1_{A_n}\|_{A(H)} = \|\n1_{A_n}\|_{A(H_n)} \\&= \|\mathcal{F}(\n1_{A_n})\|_{L^1 (\T)} = \dfrac{4}{\pi} \log (nq^{2n}p^{2n}) + O(1).
\end{split}
\end{align} This completes the proof since (\ref{amenable,subgroup}) is at most polynomial.
\end{proof}

\begin{prop}  For any $A\in SL(d,\Z)$, the  group $\Z^d\rtimes_A \Z$ has characteristic tame cuts.
\end{prop}
\begin{proof}
The proof is essentially similar to Proposition \ref{examsemidir}. Recall that the group $\Gamma = \Z^d\rtimes_A \Z$ can be identified as a subgroup of $SL(d+1,\Z)$ generated by the finite set
\begin{align*}S = \left\{
e_1 = I_{d+1} +E_{1,d+1},...,e_d = I_{d+1} +E_{d,d+1}, t= \left(\begin{matrix}
A &0\\ 0 &1
\end{matrix}\right)
\right\}.
\end{align*} It is practical to use the following unique canonical form
\[x=(v,k) \quad v\in \Z^d, k\in \Z\]
for every element $x\in \Gamma$, and in this form, the group law is given by
\begin{align*}
(v,k) (w,l) = (v+A^k w , k+l).
\end{align*}
We endow $\Gamma$ with the length function $\ell$ associated to $S$. Denote by $H$ the subgroups of $\Gamma$ generated by $\{ e_1, ...,e_d\}$. Since $\Gamma = H\rtimes_A \Z$, by Proposition \ref{polynomialext}, it is enough to show that $(H,\ell|_H)$ has characteristic tame cuts. Suppose $\ell(v,0) \leq n$. Then we can write 
\begin{align*}
(v,0) = x_1,...,x_n
\end{align*} for some $x_i = (\varepsilon_i, \delta_i)\in S^{\pm 1}$, $1\leq i \leq n$ with $\sum \delta_i = 0$ and 
\begin{align*}
v = \varepsilon_1 + A^{\delta_1} \varepsilon_2 + ... + A^{\delta_1+...+\delta_{n-1}} \varepsilon_n.
\end{align*} Notice that
\begin{align*}
\|v\|_\infty \leq \|v\|_2 \leq 1+ \|A^{\delta_1}\|+...\|A^{\delta_1 + ...+\delta_{n-1}}\| \leq n \max({\|A\|,\|A^{-1}\|})^n \leq n C^n.
\end{align*} Thus, the set $A_n = \{(v,0)\in H : v\in \Z^d, \|v\|_\infty \leq nC^n\}$ is finite and contains the relative ball $H\cap B_n$ of radius $n$. Moreover, by (\ref{hardyeq}), we have
\begin{align}\label{An}
\begin{split}
\|\n1_{A_n}\|_{M_0A(G)} &= \|\mathcal{F}(\n1_{A_n})\|_{L^1 (\T^d)} = 
\|D_{nC^n}\otimes ...\otimes D_{nC^n}\|_{L^1 (\T^d)}\\& =\|D_{nC^n}\|_{L^1 (\T)}^d =\left(\dfrac{4}{\pi} \log (nC^n) + O(1)\right)^d.
\end{split}
\end{align} Since (\ref{An}) is at most polynomial, we conclude.
\end{proof}
\begin{prop}\label{example BS}
For any $p,q\in\N$, the Baumslag-Solitar group 
\begin{align*}
BS(p,q)=\langle a,t | ta^p t^{-1}=a^q	\rangle
\end{align*} has completely bounded characteristic tame cuts.
\end{prop}
The idea of the proof is essentially the same as \cite{gal2002menability} where it is proved that Baumslag-Solitar groups are weakly amenable and a-T-menable using the following theorem.
\begin{thm}[Gal-Januszkiewicz ,\cite{gal2002menability}]\label{thmGal}
Let $p,q\in\N$, and $T$ the Bass-Serre tree of HNN-extension 
$BS(p,q) = HNN(\Z,p\Z \sim q\Z)$. Denote by $j_1:BS(p,q)\rightarrow Aut(T)$  the obvious group morphism. Denote $j_2: BS(p,q)\rightarrow \Z [\frac{1}{pq}] \rtimes_{\frac{p}{q}} \Z$ the group morphism defined by \begin{align*}
a\mapsto \left( \begin{matrix}
1 & 1\\ 0 & 1
\end{matrix} \right),\quad t\mapsto \left( \begin{matrix}
\frac{p}{q} & 0\\ 0 & 1
\end{matrix} \right).
\end{align*} 
Then the diagonal morphism $j=(j_1,j_2):BS(p,q)\rightarrow Aut(T)\times \left( \Z [\frac{1}{pq}] \rtimes_{\frac{p}{q}} \Z \right)$ is a closed embedding  (when $Aut(T)$ is endowed with the compact-open topology).
\end{thm}
We also need the following lemma.
\begin{lem}[Bo\.{z}ejko-Picardello, \cite{bozejko1993weaklyamen}]\label{treelem}
Suppose that a discrete group $\Gamma$ acts on a tree $T=(V,E)$ by isometries. Then for any base point $v_0\in V$ and $n\in\N$, the characteristic function $\phi_n$ of the set $ B_n = \{x\in \Gamma:d(v_0,xv_0)\leq n\}$ defines a completely bounded Fourier multiplier on $\Gamma$ with $\|\phi_n\|_{M_0 A}\leq 2n+1$.
\end{lem}
\begin{proof}[Proof of Proposition \ref{example BS}]
Theorem \ref{thmGal} allows us to identify $BS(p,q)$ as a discrete subgroup of the locally compact group $Aut(T)\times\left( \Z [\frac{1}{pq}] \rtimes_{\frac{p}{q}} \Z\right)$.

For all $n\in \N$, take $\phi_n\in M_0A(Aut(T)_d)$ as in Lemma \ref{treelem} for $Aut(T)_d$-action on the tree $T$, where $Aut(T)_d$ is the discrete realization of $Aut(T)$. Let $(\psi_n)_{n\in\N}$ be tame cuts of $\Z [\frac{1}{pq}] \rtimes_{\frac{p}{q}} \Z$.
For $n\in\N$, define a function $\varphi_n$ as $\varphi_n(x,y) = \phi_n(x)\psi_n(y)$ for all $x\in Aut(T)$ and $y\in \Z [\frac{1}{pq}] \rtimes_{\frac{p}{q}} \Z$. Then by Proposition \ref{restriction map}, and \cite[Lemma 1.4]{haagerup1989}, we have
\begin{align*}
\|\varphi_n\|_{M_0A(BS(p,q))} \leq \|\phi_n\|_{M_0A (Aut(T)_d)}\|\psi_n\|_{M_0A(\Z [\frac{1}{pq}] \rtimes_{\frac{p}{q}} \Z)}.
\end{align*}
Note that the right hand side grows polynomially.
Also note that the function $\varphi_n$ is characteristic since so are $\phi_n$ and $\psi_n$. It is left to prove the support of $\varphi_n$ is finite and contains the relative ball of radius $n$. For that, first we explain which length function we want to choose. As $BS(p,q)$ is finitely generated, we can choose any proper length function  on it to prove the statement. On $Aut(T)_d$, we choose the length function  defined by $\ell_1 (x) = d(v_0,xv_0)$. Since the tree $T$ is locally finite, the $Aut(T)$-action is proper, implying the length function  $\ell_1$ is proper (when $Aut(T)$ is endowed with the compact-open topology). Thus, the support of $\phi_n$ is compact. On $\Z [\frac{1}{pq}] \rtimes_{\frac{p}{q}} \Z$, we chose any word length function $\ell_2$ associated to a finite generating set. One can easily check that  the function $\ell(x,y) = \max (\ell_1(x),\ell_2(y))$ on the direct product is also a proper length function. Now it is easily seen that the support
\begin{align*}
\supp (\varphi_n) = \left(\supp (\phi_n)\times \supp(\psi_n) \right)\cap BS(p,q)
\end{align*}
is finite and contains the relative ball $\{ z\in BS(p,q):\ell (z)\leq n\}$.
\end{proof}
\begin{rmk}
Among all the examples in this section, only the following ones satisfy RD for having polynomial growth.
\begin{enumerate}[(i)]	
\item $\Z^d\rtimes_A \Z$ with $\spec (A)\subseteq S^1$.
\item $Z [\frac{1}{1}] \rtimes_{\frac{1}{1}} \Z \cong \Z^2$.
\item $BS(p,p) $ for $p\in \N$.
\end{enumerate}
Concerning non-examples, the rigidity inequality of Theorem \ref{rigidity ineq thm} shows that higher rank simple Lie groups with finite center do not have tame cuts. It follows that uniform lattices in a higher rank simple Lie group with finite center do not have completely bounded tame cuts. We do not have an example of a finitely generated group without (characteristic) tame cuts with respect to the word length function. Our best guess is the group given by the presentation 
\begin{align*}
G=\langle a,b,c\mid aba^{-1}=b^2, cac^{-1}=a^2\rangle
\end{align*} because it has a cyclic subgroup that is double exponentially distorted. The lattices $SL(3,\Z)$ and $SL(2,\Z)\ltimes \Z^2$ are also good candidates. We also conjecture that cocompact lattices in higher rank semisimple Lie groups have characteristic tame cuts. We note that the latter follows from Valette's conjecture (see \cite{Val02}).
\end{rmk}

\section*{Acknowledgements}
This work is a part of my doctoral dissertation. I am sincerely grateful to my supervisor, Indira Chatterji, for her guidance and valuable comments. I would like to acknowledge that the notion of tame cuts was inspired by Vincent Lafforgue.
\bibliographystyle{alpha}
\bibliography{mybibfile}

\begin{thebibliography}{LMR00}

\bibitem[Bo{\.{z}}82]{bozejko1981remark}
Marek Bo{\.{z}}ejko.
\newblock Remark on {H}erz-{S}chur multipliers on free groups.
\newblock {\em Math. Ann.}, 258(1):11--15, 1981/82.

\bibitem[BP93]{bozejko1993weaklyamen}
Marek Bo\.{z}ejko and Massimo~A. Picardello.
\newblock Weakly amenable groups and amalgamated products.
\newblock {\em Proc. Amer. Math. Soc.}, 117(4):1039--1046, 1993.

\bibitem[CH89]{haagerup1989}
Michael Cowling and Uffe Haagerup.
\newblock Completely bounded multipliers of the {F}ourier algebra of a simple
  {L}ie group of real rank one.
\newblock {\em Invent. Math.}, 96(3):507--549, 1989.

\bibitem[Cha17]{RDintro}
Indira Chatterji.
\newblock Introduction to the rapid decay property.
\newblock In {\em Around {L}anglands correspondences}, volume 691 of {\em
  Contemp. Math.}, pages 53--72. Amer. Math. Soc., Providence, RI, 2017.

\bibitem[DCH85]{dCH1985}
Jean De~Canni\`ere and Uffe Haagerup.
\newblock Multipliers of the {F}ourier algebras of some simple {L}ie groups and
  their discrete subgroups.
\newblock {\em Amer. J. Math.}, 107(2):455--500, 1985.

\bibitem[dL13]{deLaat2012}
Tim de~Laat.
\newblock Approximation properties for noncommutative {$L^p$}-spaces associated
  with lattices in {L}ie groups.
\newblock {\em J. Funct. Anal.}, 264(10):2300--2322, 2013.

\bibitem[Eym64]{eymard1964algebre}
Pierre Eymard.
\newblock L'alg\`ebre de {F}ourier d'un groupe localement compact.
\newblock {\em Bull. Soc. Math. France}, 92:181--236, 1964.

\bibitem[GJ03]{gal2002menability}
\'{S}wiatos\l aw~R. Gal and Tadeusz Januszkiewicz.
\newblock New a-{T}-menable {HNN}-extensions.
\newblock {\em J. Lie Theory}, 13(2):383--385, 2003.

\bibitem[Haa16]{haagerup2016}
Uffe Haagerup.
\newblock Group {$C^*$}-algebras without the completely bounded approximation
  property.
\newblock {\em J. Lie Theory}, 26(3):861--887, 2016.

\bibitem[Haa79]{haagerupanexample}
Uffe Haagerup.
\newblock An example of a nonnuclear {$C^{\ast} $}-algebra, which has the
  metric approximation property.
\newblock {\em Invent. Math.}, 50(3):279--293, 1978/79.

\bibitem[HSS10]{haagerup2010schur}
U.~Haagerup, T.~Steenstrup, and R.~Szwarc.
\newblock Schur multipliers and spherical functions on homogeneous trees.
\newblock {\em Internat. J. Math.}, 21(10):1337--1382, 2010.

\bibitem[Jol90]{jolissaint1990rapidly}
Paul Jolissaint.
\newblock Rapidly decreasing functions in reduced {$C^*$}-algebras of groups.
\newblock {\em Trans. Amer. Math. Soc.}, 317(1):167--196, 1990.

\bibitem[Kat04]{MR2039503}
Yitzhak Katznelson.
\newblock {\em An introduction to harmonic analysis}.
\newblock Cambridge Mathematical Library. Cambridge University Press,
  Cambridge, third edition, 2004.

\bibitem[Laf00]{LafforgueRD}
Vincent Lafforgue.
\newblock A proof of property ({RD}) for cocompact lattices of {${\rm
  SL}(3,\bold R)$} and {${\rm SL}(3,\bold C)$}.
\newblock {\em J. Lie Theory}, 10(2):255--267, 2000.

\bibitem[LDlS11]{lafforgue2011CBAP}
Vincent Lafforgue and Mikael De~la Salle.
\newblock Noncommutative {$L^p$}-spaces without the completely bounded
  approximation property.
\newblock {\em Duke Math. J.}, 160(1):71--116, 2011.

\bibitem[Lia16]{Lia16}
Benben Liao.
\newblock About the obstacle to proving the {B}aum-{C}onnes conjecture without
  coefficient for a non-cocompact lattice in {$Sp_4$} in a local field.
\newblock {\em J. Noncommut. Geom.}, 10(4):1243--1268, 2016.

\bibitem[LMR93]{LMR93}
Alexander Lubotzky, Shahar Mozes, and M.~S. Raghunathan.
\newblock Cyclic subgroups of exponential growth and metrics on discrete
  groups.
\newblock {\em C. R. Acad. Sci. Paris S\'{e}r. I Math.}, 317(8):735--740, 1993.

\bibitem[LMR00]{LMR00}
Alexander Lubotzky, Shahar Mozes, and M.~S. Raghunathan.
\newblock The word and {R}iemannian metrics on lattices of semisimple groups.
\newblock {\em Inst. Hautes \'{E}tudes Sci. Publ. Math.}, (91):5--53 (2001),
  2000.

\bibitem[Los84]{Losert1984}
Viktor Losert.
\newblock Properties of the {F}ourier algebra that are equivalent to
  amenability.
\newblock {\em Proc. Amer. Math. Soc.}, 92(3):347--354, 1984.

\bibitem[MPS81]{1981hardyineq}
O.~Carruth McGehee, Louis Pigno, and Brent Smith.
\newblock Hardy's inequality and the {$L^{1}$} norm of exponential sums.
\newblock {\em Ann. of Math. (2)}, 113(3):613--618, 1981.

\bibitem[Neb82]{Nebbia1982}
Claudio Nebbia.
\newblock Multipliers and asymptotic behaviour of the {F}ourier algebra of
  nonamenable groups.
\newblock {\em Proc. Amer. Math. Soc.}, 84(4):549--554, 1982.

\bibitem[Val02]{Val02}
Alain Valette.
\newblock {\em Introduction to the {B}aum-{C}onnes conjecture}.
\newblock Lectures in Mathematics ETH Z\"{u}rich. Birkh\"{a}user Verlag, Basel,
  2002.
\newblock From notes taken by Indira Chatterji, With an appendix by Guido
  Mislin.

\end{thebibliography}
\end{document}